\documentclass[11pt]{amsart}

\usepackage[a4paper,top=2cm,bottom=2cm,left=2cm,right=2cm]{geometry}

\usepackage[T1]{fontenc}
\usepackage{amsmath,amssymb,amsthm}
\usepackage{graphicx}
\usepackage{subcaption}
\usepackage{float}
\usepackage{hyperref}

\newtheorem{lemma}{Lemma}[section]

\newtheorem{theorem}{Theorem}[section]
\newtheorem{proposition}{Proposition}[section]
\newtheorem{corollary}{Corollary}[section]

\newtheorem{remark}{Remark}[section]

\theoremstyle{definition}

\numberwithin{equation}{section}
\title[Driven--Damped $\phi^4$ Equation on Bounded Domains]{Time-Periodic Dynamics of a Driven--Damped $\phi^4$ Equation on Bounded Domains}

\author{Vassilios M Rothos}
\address{School of Mechanical Engineering, Faculty of Engineering\\
Aristotle University of Thessaloniki\\
Thessaloniki 54124, Greece}
\email{rothos@auth.gr}
\thanks{Corresponding author. Email: rothos@auth.gr}
\thanks{The research was funded by Aristotle University of Thessaloniki (AUTH) Research Council grants number 73191, 73699, 11682.}



\date{}

\dedicatory{}

\begin{document}

\begin{abstract}
We prove the existence of time-periodic solutions for a driven--damped $\phi^4$
wave equation posed on a bounded spatial domain with periodic boundary conditions.
The result is obtained at the level of the full partial differential equation.
The analysis is based on a Lyapunov--Schmidt decomposition in a time-periodic
function space. A key structural feature is the presence of linear damping,
which ensures uniform invertibility of the associated linear operator on the
complement of the kernel and removes the need for nonresonance conditions.
After introducing a truncation to control the cubic nonlinearity, we solve the
complement equation via a contraction mapping argument and reduce the problem
to a scalar kernel equation. The latter is treated using a mean/zero-mean
decomposition and a fixed-point argument. A period-averaged energy estimate
then allows one to remove the truncation and obtain a time-periodic weak
solution under explicit smallness conditions on the forcing parameters
\end{abstract}

\maketitle

\section{Introduction}
\label{sec0}

Driven and damped nonlinear wave equations arise in a wide range of physical
contexts, including condensed matter physics, nonlinear optics, and classical
field theory \cite{Birnir1994}. Among the simplest and most extensively studied
examples is the $\phi^4$ model, whose double-well potential supports both
homogeneous equilibria and topological kink solutions
\cite{CampbellSchonfeldWingate1983,Manton2021}. In the presence of external
forcing, dissipation, and spatial or spatio-temporal inhomogeneities, the
resulting dynamics becomes genuinely nonintegrable and may exhibit rich
time-dependent behavior, including frequency locking, resonance phenomena, and
time-periodic responses
\cite{GatlikDobrowolskiKevrekidis2024,GatlikDobrowolskiCaputoKevrekidis2026}.
Such effects have been widely observed in numerical simulations and in
collective-coordinate descriptions of driven $\phi^4$ and sine--Gordon systems
\cite{CampbellSchonfeldWingate1983,KivsharMalomed1989,Manton2021}.

\textbf{Related works.}
From a mathematical perspective, the existence of time-periodic solutions for
nonlinear wave equations has been studied using functional-analytic and
operator-theoretic methods
\cite{Pazy1983,LionsMagenes1972}. A particularly effective approach is based on
Lyapunov--Schmidt reduction for second-order evolution equations with periodic
forcing, as developed by Fe\v{c}kan and others
\cite{Feckan1998,HaleRaugel1992,ChueshovLasiecka2010}. This framework relies on
a decomposition with respect to the kernel of a self-adjoint spatial operator,
Fourier analysis in time, and uniform control of the linearized dynamics on the
orthogonal complement, reducing the problem to a finite-dimensional one.

For nonlinear field equations posed on the real line, however, several
fundamental obstacles limit the direct applicability of such methods. The
associated spatial operator typically has continuous spectrum, which prevents a
straightforward Fourier--spectral inversion, while the cubic nonlinearity of the
$\phi^4$ model is not globally Lipschitz on natural energy spaces. In addition,
dispersive effects and radiation mechanisms interact nontrivially with coherent
structures; see, for example,
\cite{SofferWeinstein1999,KowalczykMartelMunoz2017}. As a consequence,
PDE-level existence results for time-periodic solutions in this setting remain
difficult to obtain.

Analytical investigations therefore often rely on reduced descriptions, most
notably collective-coordinate approaches, in which the dynamics is projected
onto a finite set of modes associated with a kink manifold
\cite{KivsharMalomed1989,Manton2021,GatlikDobrowolskiCaputoKevrekidis2026}.
While these methods provide valuable physical insight, they are not, in general,
derived from a fully controlled analysis of the underlying PDE.

In contrast, numerical simulations of driven--damped $\phi^4$ equations are
necessarily carried out on finite spatial intervals, supplemented with boundary
conditions such as periodic or Dirichlet conditions
\cite{GatlikDobrowolskiKevrekidis2024,GatlikDobrowolskiCaputoKevrekidis2026}.
This replaces the continuous spectrum of the real-line problem by a discrete
one. From an analytical viewpoint, this is not merely a numerical artifact but a
structural modification that restores compactness and enables spectral
decompositions unavailable on $\mathbb{R}$. The present work is motivated by the
observation that this bounded-domain setting is precisely the regime in which a
Lyapunov--Schmidt strategy can be implemented at the level of the full PDE.

\textbf{Main results.}
In this paper we consider a driven--damped $\phi^4$ field equation posed on a
bounded spatial domain with periodic boundary conditions, incorporating small
spatial and spatio-temporal modulations. Our main result establishes the
existence of time-periodic solutions of the full partial differential equation
under explicit smallness conditions on the forcing parameters.

The analysis combines operator-theoretic and nonlinear techniques within a
Lyapunov--Schmidt framework. The problem is formulated on a finite periodic
domain, so that the associated elliptic operator has discrete spectrum with a
finite-dimensional kernel. A truncation is introduced to render the cubic
nonlinearity globally Lipschitz, together with explicit Sobolev-based bounds.

A Lyapunov--Schmidt decomposition reduces the problem to a coupled system
consisting of a complement equation and a finite-dimensional kernel equation.
A key feature of the analysis is that linear damping ensures uniform
invertibility of the associated time-periodic linear operator on the orthogonal
complement of the kernel. This eliminates the need for nonresonance conditions
and avoids small-divisor difficulties that arise in conservative settings.

The reduced kernel equation is then solved using a mean/zero-mean decomposition
and a fixed-point argument. Finally, a period-averaged energy identity yields
an a priori estimate that allows one to remove the truncation and obtain a
time-periodic solution of the original equation.

The novelty of the present work lies in the identification of a structural
mechanism combining spectral discreteness on bounded domains with
damping-induced invertibility of the associated time-periodic linear operator.
More precisely:
(i) we obtain an existence result at the level of the full partial differential
equation without imposing nonresonance conditions on the temporal frequency;
(ii) we show that linear damping provides a uniform lower bound on the temporal
Fourier multipliers, thereby eliminating small-divisor obstructions; and
(iii) we carry out a fully controlled Lyapunov--Schmidt reduction in which the
infinite-dimensional dynamics is reduced to a scalar kernel equation.
To the best of our knowledge, a result of this type for driven--damped
$\phi^4$ equations at the PDE level has not previously appeared in the
literature.

\textbf{Organization of the paper.}
The paper is organized as follows.
In Section~\ref{sec1} we introduce the model and the functional setting.
Section~\ref{sec2} establishes global Lipschitz estimates for the truncated
nonlinearity.
In Section~\ref{sec3} we perform the Lyapunov--Schmidt decomposition and prove
invertibility of the linear operator on the complement.
Section~\ref{sec4} solves the complement equation and derives the reduced kernel
equation.
In Section~\ref{sec5} we solve the reduced equation and obtain existence of
time-periodic solutions.
Finally, Sections~\ref{sec6} and \ref{sec_conclusions} discuss implications,
limitations, and possible extensions.
\section{Model and functional setting}
\label{sec1}

\subsection{Driven--damped $\phi^4$ model and periodic formulation}
\label{sec1a}

The driven--damped $\phi^4$ equation considered here serves as a prototypical
continuum model for the dynamics of interfaces and domain walls in spatially
inhomogeneous and externally forced media. Closely related field equations
arise, for instance, in the theory of long Josephson junctions with geometric or
material modulations, as well as in nonlinear waveguides and patterned media,
where spatial variations of effective stiffness or dispersion lead naturally to
position-dependent coefficients in the governing equations. Periodic
modulations of this type may represent engineered geometries, layered
substrates, or externally imposed lattices, and are known to induce effective
pinning, resonant responses, and nontrivial transport phenomena at the level of
kink dynamics.

A substantial body of work has investigated such effects using numerical
simulations of the full field equation or reduced collective-coordinate models,
in which the infinite-dimensional dynamics is projected onto a finite set of
effective degrees of freedom associated with the kink position, width, or
internal modes; see, for example,
\cite{GatlikDobrowolskiCaputoKevrekidis2026,GatlikDobrowolskiKevrekidis2024}
and references therein. While these approaches provide valuable physical
insight, they do not address the existence of time-periodic solutions at the
level of the full partial differential equation. The formulation adopted in
this work bridges this gap by treating the driven--damped $\phi^4$ model
directly as a nonlinear evolution equation on a bounded domain and by
establishing, within a rigorous operator-theoretic framework, the existence of
genuine time-periodic solutions under small spatio-temporal modulations.

We consider the driven--damped $\phi^4$ field equation
\begin{equation}
\partial_t^2 \phi + \eta \partial_t \phi
- \partial_x\!\big(F(x)\partial_x \phi\big)
+ \lambda(t,x)\,\phi(\phi^2-1)
= -\Gamma ,
\qquad (t,x)\in \mathbb{R}\times\mathbb{R},
\label{eq:pde_R}
\end{equation}
with $\eta>0$ (linear damping), constant bias $\Gamma\in\mathbb{R}$, and
modulations
\begin{equation}
F(x)=1+\varepsilon_1\sin(qx),\qquad 
\lambda(t,x)=1+\varepsilon_2\sin(kx-\omega t),
\qquad 0<\varepsilon_1,\varepsilon_2\ll 1.
\label{eq:mods_R}
\end{equation}

On $\mathbb{R}$ one typically focuses on kink dynamics. However, numerical PDE
computations are necessarily carried out on finite spatial intervals with
boundary conditions, which discretize the spatial spectrum and place the
problem within an operator-theoretic framework compatible with
Lyapunov--Schmidt reduction \cite{Feckan1998}. Motivated by this observation,
we formulate and prove Lyapunov–Schmidt framework of Feckan under periodic
boundary conditions on a finite domain. This setting does not represent a
single topological kink on a circle (which is topologically incompatible with
periodic boundary conditions), but it provides a clean PDE-level framework in
which time-periodic responses can be analyzed rigorously, while retaining the
spectral discreteness characteristic of finite-domain simulations.

Fix $L>0$ and let $\mathbb{T}_L:=\mathbb{R}/(L\mathbb{Z})$ be the one-dimensional
torus. Let
\[
T:=\frac{2\pi}{\omega},\qquad \mathbb{S}_T:=\mathbb{R}/(T\mathbb{Z}),
\]
and seek solutions $\phi(t,\cdot)\in H^1(\mathbb{T}_L)$ that are $T$-periodic in
time and periodic in space:
\[
\phi(t+T,x)=\phi(t,x),\qquad 
\phi(t,x+L)=\phi(t,x),\qquad 
\phi_x(t,x+L)=\phi_x(t,x).
\]

We introduce the Hilbert spaces
\[
Y:=L^2(\mathbb{T}_L),\qquad 
X:=H^1(\mathbb{T}_L)\hookrightarrow Y,
\]
endowed with their standard norms $\|\cdot\|_Y$ and $\|\cdot\|_X$. Assuming
$|\varepsilon_1|<1$, we have
\[
F(x)\ge F_0:=1-|\varepsilon_1|>0 \quad \text{for all }x\in\mathbb{T}_L.
\]

Define the self-adjoint operator
\begin{equation}
A u:= -\partial_x\!\big(F(x)u_x\big),\qquad 
\mathrm{dom}(A)=H^2(\mathbb{T}_L).
\label{eq:A_def}
\end{equation}
Then $A$ is nonnegative, self-adjoint on $Y$, and has compact resolvent.
Consequently, it admits an orthonormal eigenbasis $\{u_j\}_{j\ge 0}\subset Y$
with eigenvalues $0=\lambda_0<\lambda_1\le\lambda_2\le\cdots$ and
$\lambda_j\to\infty$.

A key structural feature, crucial for the Lyapunov--Schmidt decomposition, is
the explicit characterization of the kernel:
\begin{equation}
\ker A=\{\text{constants on }\mathbb{T}_L\}
=\mathrm{span}\{w_0\},\qquad 
w_0(x):=L^{-1/2}.
\label{eq:kerA}
\end{equation}
Indeed, $Au=0$ implies $(F u_x)_x=0$ in distributions, so $F u_x\equiv c$;
integrating over one period yields $c=0$, hence $u_x\equiv 0$.

Let $P:Y\to\ker A$ be the orthogonal projector $Pu=\langle u,w_0\rangle w_0$
and set $Q:=I-P$.

With $u:=\phi$, equation \eqref{eq:pde_R} on
$\mathbb{S}_T\times\mathbb{T}_L$ can be written in abstract form as
\begin{equation}
u_{tt}+\eta u_t + A u = f(u,t),
\label{eq:abstract}
\end{equation}
where
\begin{equation}
f(u,t)(x):=-\lambda(t,x)\,u(x)\big(u(x)^2-1\big)-\Gamma.
\label{eq:f_def}
\end{equation}

The cubic growth of the nonlinearity prevents $f(\cdot,t)$ from being globally
Lipschitz as a map $X\to Y$, while the Lyapunov--Schmidt framework of
\cite{Feckan1998} requires such a property in order to apply a contraction
mapping argument on the complement equation. We therefore introduce a standard
cut-off, solve the truncated problem, and subsequently remove the cut-off by
means of an a priori estimate.

The bounded periodic framework introduced above plays a structural role in the
analysis. We briefly contrast it with the unbounded setting.

\subsection{Bounded domain versus the real line}
\label{sec1b}

The analysis is carried out on the bounded periodic domain $\mathbb{T}_L$, where
the associated spatial operator has compact resolvent and purely discrete
spectrum.

In contrast, on the real line $\mathbb{R}$, the linearized operator around
coherent structures typically exhibits continuous spectrum in addition to
discrete eigenvalues, and no spectral gap separates zero from the continuous
spectrum. As a result, invertibility properties analogous to those used in the
Lyapunov--Schmidt reduction generally fail, or require weighted spaces and
substantially different techniques.

Moreover, on $\mathbb{R}$ dispersive effects and radiation mechanisms interact
with internal modes in a fundamentally different way, and such mechanisms are
absent in the bounded periodic setting considered here. For this reason, the
existence result established in the present work should be viewed as genuinely
bounded-domain and does not constitute a direct approximation of the
unbounded problem.
\section{Truncation and nonlinear estimates}
\label{sec2}

In order to apply a Lyapunov--Schmidt reduction to the complement equation, we
require a nonlinearity that is globally Lipschitz as a map from $X$ to $Y$.
The cubic term in \eqref{eq:f_def} does not satisfy this property on all of
$H^1(\mathbb{T}_L)$, so we introduce a standard smooth truncation. Such
truncation procedures are classical in the analysis of nonlinear evolution
equations in energy spaces; see, for instance,
\cite{LionsMagenes1972,Temam1997}.

Let $\chi\in C^\infty(\mathbb{R}_+;[0,1])$ be such that
\[
\chi(r)=1 \quad \text{for } 0\le r\le 1,
\qquad
\chi(r)=0 \quad \text{for } r\ge 2.
\]
For $R>0$, define the truncated nonlinearity by
\begin{equation}
f_R(u,t):=\chi\!\left(\frac{\|u\|_X}{R}\right)\,f(u,t).
\label{eq:cutoff}
\end{equation}
Thus, the cut-off preserves the original nonlinearity on the ball
$\{\|u\|_X\le R\}$ and enforces global Lipschitz continuity outside a slightly
larger ball.

\begin{proposition}[Global Lipschitz property of the truncated nonlinearity]
\label{prop:Lip_fR}
Let $L>0$ and set $\mathbb{T}_L:=\mathbb{R}/(L\mathbb{Z})$.
Let $X:=H^1(\mathbb{T}_L)$ and $Y:=L^2(\mathbb{T}_L)$.
Assume that $\lambda\in L^\infty(\mathbb{S}_T\times\mathbb{T}_L)$ and
$\Gamma\in\mathbb{R}$, and let $f_R$ be defined by \eqref{eq:cutoff}.
Then, for each $t\in\mathbb{S}_T$, the map $f_R(\cdot,t):X\to Y$ is globally
Lipschitz, uniformly in $t$. More precisely, there exists a constant $M_R>0$,
depending only on
\[
R,\ \|\lambda\|_{L^\infty},\ |\Gamma|,\ L,\ \|\chi'\|_{L^\infty},
\]
such that, for all $u_1,u_2\in X$ and all $t\in\mathbb{S}_T$,
\begin{equation}
\|f_R(u_1,t)-f_R(u_2,t)\|_Y \le M_R \|u_1-u_2\|_X.
\label{eq:Lip_fR}
\end{equation}
Moreover, $M_R$ may be chosen explicitly in terms of the Sobolev embedding
constants on $\mathbb{T}_L$.
\end{proposition}

\begin{proof}
Fix $u_1,u_2\in X$ and $t\in\mathbb{S}_T$. By definition,
\[
f_R(u_1,t)-f_R(u_2,t)
=
\chi\!\left(\frac{\|u_1\|_X}{R}\right)\bigl(f(u_1,t)-f(u_2,t)\bigr)
+
\Bigl[
\chi\!\left(\frac{\|u_1\|_X}{R}\right)
-
\chi\!\left(\frac{\|u_2\|_X}{R}\right)
\Bigr]f(u_2,t).
\]
We estimate the two terms separately.

\medskip

\noindent
\emph{Step 1: Estimate of the difference $f(u_1,t)-f(u_2,t)$.}
Let $g(z)=z(z^2-1)=z^3-z$. Then
\[
f(u,t)(x)=-\lambda(t,x)g(u(x))-\Gamma,
\]
so that
\[
f(u_1,t)-f(u_2,t)
=
-\lambda(t,\cdot)\bigl(g(u_1)-g(u_2)\bigr).
\]
A direct computation shows that
\[
|g(a)-g(b)|
\le (|a|^2+|a||b|+|b|^2+1)|a-b|.
\]
Using the Sobolev embedding $H^1(\mathbb{T}_L)\hookrightarrow L^\infty(\mathbb{T}_L)$,
we obtain
\[
\|f(u_1,t)-f(u_2,t)\|_Y
\le
\|\lambda\|_{L^\infty}(1+3K^2)\|u_1-u_2\|_X,
\]
where
\[
K=C_\infty(L)\max\{\|u_1\|_X,\|u_2\|_X\}.
\]

If at least one cut-off factor is nonzero, then $\|u_i\|_X\le 2R$ and hence
\[
K\le 2C_\infty(L)R.
\]
Therefore,
\[
\|f(u_1,t)-f(u_2,t)\|_Y
\le
C_1(R)\|u_1-u_2\|_X,
\]
for a constant $C_1(R)>0$ depending only on $R$, $\|\lambda\|_{L^\infty}$,
and $L$.

\medskip

\noindent
\emph{Step 2: Estimate of the cut-off variation term.}
By the mean value theorem,
\[
\left|
\chi\!\left(\frac{\|u_1\|_X}{R}\right)
-
\chi\!\left(\frac{\|u_2\|_X}{R}\right)
\right|
\le
\frac{\|\chi'\|_{L^\infty}}{R}\,\|u_1-u_2\|_X.
\]
Thus,
\[
\left\|
\Bigl[
\chi\!\left(\frac{\|u_1\|_X}{R}\right)
-
\chi\!\left(\frac{\|u_2\|_X}{R}\right)
\Bigr]f(u_2,t)
\right\|_Y
\le
\frac{\|\chi'\|_{L^\infty}}{R}\,\|f(u_2,t)\|_Y\,\|u_1-u_2\|_X.
\]

If $\chi(\|u_2\|_X/R)\neq 0$, then $\|u_2\|_X\le 2R$, and by Sobolev
embeddings,
\[
\|f(u_2,t)\|_Y
\le
C_2(R)\bigl(1+|\Gamma|\bigr),
\]
for a constant $C_2(R)>0$ depending only on $R$, $\|\lambda\|_{L^\infty}$,
and $L$.

\medskip

\noindent
\emph{Conclusion.}
Combining the above estimates, we obtain
\[
\|f_R(u_1,t)-f_R(u_2,t)\|_Y
\le
M_R\|u_1-u_2\|_X,
\]
for a suitable constant $M_R>0$ depending only on
\[
R,\ \|\lambda\|_{L^\infty},\ |\Gamma|,\ L,\ \|\chi'\|_{L^\infty},
\]
and the Sobolev embedding constants on $\mathbb{T}_L$.
This proves \eqref{eq:Lip_fR}.
\end{proof}

Similar Lipschitz regularizations are commonly used in the study of semilinear
wave equations in order to enable fixed-point arguments; see, for example,
\cite{Pazy1983}.
\section{Lyapunov--Schmidt decomposition and complement invertibility}
\label{sec3}

Let $L^2(\mathbb{S}_T;Y)$ denote the Bochner space of $T$-periodic $Y$-valued
functions, equipped with the norm
\[
\|v\|_{L^2_T(Y)}^2:=\int_0^T \|v(t)\|_Y^2\,dt.
\]
The temporal Fourier functions
\[
e_n(t):=e^{i2\pi nt/T}, \qquad n\in\mathbb{Z},
\]
form an orthonormal basis of $L^2(\mathbb{S}_T)$, while
$\{u_j\}_{j\ge0}$ is an orthonormal basis of $Y$.
Consequently, the family $\{e_nu_j\}_{n\in\mathbb{Z},\,j\ge0}$ forms an
orthonormal basis of $L^2(\mathbb{S}_T;Y)$.

Throughout this section we work with the truncated problem
\begin{equation}
u_{tt}+\eta u_t + A u = f_R(u,t),
\qquad u(t+T)=u(t),
\label{eq:truncated}
\end{equation}
and postpone removal of the cut-off to the a priori estimates established later.

\subsection{Lyapunov--Schmidt decomposition and reduction strategy}

We decompose the unknown as
\[
u(t)=v(t)+w(t),\qquad
v(t):=Pu(t)\in\ker A,\qquad
w(t):=Qu(t)\in\ker A^\perp,
\]
where $P:Y\to\ker A$ is the orthogonal projector onto the kernel of $A$, and
$Q:=I-P$ is the complementary projector.
Since $\ker A$ is one-dimensional and spanned by the normalized constant
function $w_0(x)=L^{-1/2}$, there exists a scalar function
$x:\mathbb{S}_T\to\mathbb{R}$ such that
\begin{equation}
v(t)=x(t)\,w_0.
\label{eq:v_representation}
\end{equation}
Moreover,
\[
x(t)=\langle u(t),w_0\rangle_{L^2(\mathbb{T}_L)}
=\frac{1}{\sqrt{L}}\int_{\mathbb{T}_L}u(t,x)\,dx,
\]
so $x(t)$ is precisely the spatial average of the solution. It describes the
component of the dynamics along the neutral zero eigenmode of the spatial
operator $A$.

Applying the projections $P$ and $Q$ to \eqref{eq:truncated}, and using the
identities $PA=0$ and $QA=AQ$ on $\mathrm{dom}(A)$, we obtain the coupled system
\begin{align}
v_{tt}(t)+\eta v_t(t)
&= P f_R\bigl(v(t)+w(t),t\bigr),
\label{eq:P_eq_detailed}\\
w_{tt}(t)+\eta w_t(t)+A w(t)
&= Q f_R\bigl(v(t)+w(t),t\bigr).
\label{eq:Q_eq_detailed}
\end{align}
Equation \eqref{eq:P_eq_detailed} governs the kernel component, whereas
\eqref{eq:Q_eq_detailed} governs the dynamics on the orthogonal complement.

Following the Lyapunov--Schmidt strategy of Fe\v{c}kan \cite{Feckan1998}, we
treat \eqref{eq:Q_eq_detailed} as a linear time-periodic equation for $w$, with
a nonlinear right-hand side depending parametrically on $v$. To this end, we
introduce the linear operator
\[
\mathcal{L}w:=w_{tt}+\eta w_t+Aw,
\]
acting on $T$-periodic functions with values in $\ker A^\perp$.
As shown below, $\mathcal{L}$ is invertible on this space and its inverse is
uniformly bounded. Consequently, \eqref{eq:Q_eq_detailed} can be rewritten in
fixed-point form as
\begin{equation}
w=\mathcal{L}^{-1}Q f_R(v+w,t).
\label{eq:Q_fixed_point}
\end{equation}

Since the truncated nonlinearity $f_R$ is globally Lipschitz as a map
$X\to Y$, the right-hand side of \eqref{eq:Q_fixed_point} is a contraction on a
suitable $T$-periodic function space provided the Lipschitz constant $M_R$ is
sufficiently small relative to the inverse bound for $\mathcal{L}$. Under this
condition, the contraction mapping principle yields, for each prescribed
$T$-periodic kernel component $v$, a unique $T$-periodic complement component
$w=w(v)$ solving \eqref{eq:Q_eq_detailed}. Moreover, the dependence of $w$ on
$v$ is Lipschitz continuous.

Substituting the resulting complement component $w(v)$ into the kernel equation
\eqref{eq:P_eq_detailed} reduces the infinite-dimensional PDE to a
finite-dimensional equation for the scalar function $x(t)$ in
\eqref{eq:v_representation}. This reduced equation captures the dynamics along
the neutral mode, while the higher modes are slaved to it through the map
$v\mapsto w(v)$.

Once a $T$-periodic solution $x(t)$ of the reduced kernel equation has been
constructed, the corresponding $T$-periodic solution of the truncated PDE is
recovered as
\[
u(t)=x(t)w_0+w(x(t)w_0).
\]
The final step is then to show that this solution remains in the region where
the cut-off is inactive. This allows the truncation to be removed and yields a
genuine $T$-periodic solution of the original equation.

\subsection{Invertibility estimate on the complement}

We now study the linear operator
\begin{equation}
\mathcal{L}w:=w_{tt}+\eta w_t+Aw,
\label{eq:L_def_repeat}
\end{equation}
acting on $T$-periodic functions $w:\mathbb{S}_T\to Y$ with values in
$\ker A^\perp\subset Y$.

Since $A$ is self-adjoint with compact resolvent on $Y=L^2(\mathbb{T}_L)$, it
admits a complete orthonormal basis of eigenfunctions
$\{u_j\}_{j\ge0}$ with eigenvalues
\[
0=\lambda_0<\lambda_1\le\lambda_2\le\cdots,
\qquad \lambda_j\to\infty,
\]
and $\ker A=\mathrm{span}\{u_0\}$.
Thus the restriction $w\in\ker A^\perp$ means precisely that only modes with
$j\ge1$ appear in the expansion of $w$.

Expanding $w$ in the time--space Fourier basis
$\{e_n(t)u_j(x)\}_{n\in\mathbb{Z},\,j\ge1}$, where
$e_n(t)=\exp(i2\pi nt/T)$, a direct computation shows that each such mode is an
eigenfunction of $\mathcal{L}$:
\[
\mathcal{L}(e_nu_j)
=
\left(
-\left(\frac{2\pi n}{T}\right)^2
+i\eta\frac{2\pi n}{T}
+\lambda_j
\right)e_nu_j
=:\Delta_{n,j}\,e_nu_j.
\]
Hence $\mathcal{L}$ is diagonal in this basis, and invertibility reduces to a
uniform lower bound on the complex multipliers $\Delta_{n,j}$. Such spectral
decompositions are classical in the analysis of nonlinear wave equations with
discrete spectrum; see also \cite{KapitulaPromislow2013}.

In the undamped case $\eta=0$, one encounters the usual small-divisor problem,
since the real quantities $\lambda_j-(2\pi n/T)^2$ may become arbitrarily small
unless one imposes a nonresonance condition, as in \cite{Feckan1998}. In the
present damped setting, however, the imaginary part introduced by the term
$\eta w_t$ regularizes all nonzero temporal harmonics.

\begin{lemma}[Uniform lower bound for the damped multipliers]
\label{lem:Delta_bound}
Assume $\eta>0$. Then for every $j\ge1$ and every $n\in\mathbb{Z}$,
\begin{equation}
|\Delta_{n,j}|\ge
\begin{cases}
\lambda_1, & n=0,\\[1ex]
\displaystyle \eta\,\frac{2\pi|n|}{T}, & n\neq0.
\end{cases}
\label{eq:Delta_bound_detailed}
\end{equation}
Consequently,
\begin{equation}
c_\eta:=\min\!\left\{\lambda_1,\frac{2\pi\eta}{T}\right\}>0,
\qquad
\inf_{j\ge1,\;n\in\mathbb{Z}}|\Delta_{n,j}|\ge c_\eta .
\label{eq:ceta}
\end{equation}
In particular,
\[
\|\mathcal{L}^{-1}\|\le c_\eta^{-1}.
\]
\end{lemma}

\begin{proof}
We distinguish between the cases $n=0$ and $n\neq0$.

If $n=0$, then
\[
\Delta_{0,j}=\lambda_j.
\]
Since $j\ge1$, we have $\lambda_j\ge\lambda_1$, and therefore
\[
|\Delta_{0,j}|=\lambda_j\ge\lambda_1.
\]

If $n\neq0$, then
\[
\Delta_{n,j}
=
\lambda_j-\left(\frac{2\pi n}{T}\right)^2
+i\eta\left(\frac{2\pi n}{T}\right).
\]
For any complex number $z=a+ib$, one has $|z|\ge|b|$. Hence
\[
|\Delta_{n,j}|
\ge
\bigl|\Im\Delta_{n,j}\bigr|
=
\eta\,\frac{2\pi|n|}{T}
\ge
\eta\,\frac{2\pi}{T}.
\]
Combining the two cases gives \eqref{eq:Delta_bound_detailed}, and
\eqref{eq:ceta} follows immediately. Since $\eta>0$ and $\lambda_1>0$, we have
$c_\eta>0$. The bound for $\|\mathcal{L}^{-1}\|$ is then an immediate
consequence of the uniform lower bound on the multipliers.
\end{proof}

Lemma~\ref{lem:Delta_bound} shows that damping removes the small-divisor obstruction: the modes with $n\neq0$ are controlled by the
imaginary part generated by the damping term, while the purely spatial modes
$n=0$ are controlled by the spectral gap $\lambda_1$ of $A$. This mechanism is
characteristic of dissipative hyperbolic equations, where damping induces
spectral separation and prevents small-divisor phenomena; compare
\cite{HaleRaugel1992,ChueshovLasiecka2010}.

\begin{lemma}[Bounded inverse on $L^2_T(Y)$ for the complement equation]
\label{lem:L_inverse}
Let $\eta>0$ and let $h\in L^2(\mathbb{S}_T;Y)$ satisfy $Ph=0$.
Then there exists a unique $T$-periodic solution
$w\in L^2(\mathbb{S}_T;Y)$ with $Pw=0$ of the equation
\[
\mathcal{L}w=h,
\]
and this solution satisfies the estimate
\begin{equation}
\|w\|_{L^2_T(Y)}\le\frac{1}{c_\eta}\,\|h\|_{L^2_T(Y)}.
\label{eq:Linverse_bound_detailed}
\end{equation}
\end{lemma}

\begin{proof}
Since $Ph=0$, the function $h$ has no component in $\ker A$ and therefore
admits a Fourier expansion involving only spatial modes with index $j\ge1$:
\[
h(t,x)
=
\sum_{n\in\mathbb{Z}}\sum_{j\ge1}\widehat h_{n,j}\,e_n(t)u_j(x),
\qquad
\widehat h_{n,j}\in\mathbb{C}.
\]
We seek a solution $w$ of the same form,
\[
w(t,x)
=
\sum_{n\in\mathbb{Z}}\sum_{j\ge1}\widehat w_{n,j}\,e_n(t)u_j(x).
\]
Substituting into $\mathcal{L}w=h$ and using the diagonal action of
$\mathcal{L}$ on the basis elements $e_nu_j$, we obtain
\[
\Delta_{n,j}\,\widehat w_{n,j}=\widehat h_{n,j}
\qquad \text{for all } (n,j).
\]
By Lemma~\ref{lem:Delta_bound}, $|\Delta_{n,j}|\ge c_\eta>0$, and therefore
\[
\widehat w_{n,j}=\frac{\widehat h_{n,j}}{\Delta_{n,j}}
\]
is uniquely defined. This yields existence and uniqueness of a $T$-periodic
solution $w$ with $Pw=0$.

For the norm estimate, Parseval's identity gives
\[
\|w\|_{L^2_T(Y)}^2
=
T\sum_{n\in\mathbb{Z}}\sum_{j\ge1}|\widehat w_{n,j}|^2
=
T\sum_{n\in\mathbb{Z}}\sum_{j\ge1}
\frac{|\widehat h_{n,j}|^2}{|\Delta_{n,j}|^2}.
\]
Using again the lower bound $|\Delta_{n,j}|\ge c_\eta$, we obtain
\[
\|w\|_{L^2_T(Y)}^2
\le
\frac{T}{c_\eta^2}
\sum_{n\in\mathbb{Z}}\sum_{j\ge1}|\widehat h_{n,j}|^2
=
\frac{1}{c_\eta^2}\|h\|_{L^2_T(Y)}^2.
\]
Taking square roots yields \eqref{eq:Linverse_bound_detailed}.
\end{proof}

\begin{remark}[Dependence on the spatial period]
\label{rem:L_dependence}
The constant $c_\eta$ in \eqref{eq:Linverse_bound_detailed} depends on the
spatial period $L$ through the first positive eigenvalue of the spatial
operator on $\mathbb{T}_L$. In particular, for the Laplace operator one has
\[
\lambda_1(L)\sim \mathcal{O}(L^{-2})
\qquad \text{as } L\to\infty,
\]
and hence $c_\eta\to0$ as $L\to\infty$. Therefore the inverse bound
\eqref{eq:Linverse_bound_detailed} is not uniform in $L$. This makes clear that
the present Lyapunov--Schmidt argument relies essentially on the bounded
periodic setting and does not extend directly to the unbounded domain
$\mathbb{R}$.
\end{remark}
\section{Solving the complement equation by contraction and deriving the kernel equation}
\label{sec4}
\normalsize
In this section we solve the complement equation arising from the
Lyapunov--Schmidt decomposition and derive the reduced finite-dimensional
equation governing the kernel component. This is the step at which the
infinite-dimensional PDE is reduced to an effective scalar equation. The
argument follows the general strategy of Fe\v{c}kan \cite{Feckan1998}, adapted
here to the present dissipative and truncated setting.

Let
\[
v\in L^2(\mathbb{S}_T;\ker A)
\]
be fixed. We consider the complement equation
\begin{equation}
w_{tt}+\eta w_t+Aw = Q f_R(v+w,t),
\qquad
w(t)\in\ker A^\perp
\quad \text{for a.e. } t\in[0,T].
\label{eq:Q_eq_repeat}
\end{equation}
By Section~\ref{sec3}, the linear operator
\[
\mathcal{L}:=\partial_t^2+\eta\partial_t+A
\]
is invertible on the space of $T$-periodic functions with values in
$\ker A^\perp$, and its inverse satisfies
\begin{equation}
\|\mathcal{L}^{-1}h\|_{L^2_T(Y)}
\le
\frac{1}{c_\eta}\,\|h\|_{L^2_T(Y)}.
\label{eq:L_inverse_sec4}
\end{equation}
Accordingly, \eqref{eq:Q_eq_repeat} may be rewritten as the fixed-point problem
\begin{equation}
w=\Phi_v(w),
\qquad
\Phi_v(w):=\mathcal{L}^{-1}Q f_R(v+w,t).
\label{eq:fixed_point_complement}
\end{equation}
We regard \eqref{eq:fixed_point_complement} as an equation on the Banach space
$L^2(\mathbb{S}_T;\ker A^\perp)$ endowed with the norm $\|\cdot\|_{L^2_T(Y)}$.

We next show that $\Phi_v$ is a contraction. Let
$w_1,w_2\in L^2(\mathbb{S}_T;\ker A^\perp)$. Using the boundedness of the
orthogonal projection $Q$ on $Y$, the inverse estimate
\eqref{eq:L_inverse_sec4}, and the global Lipschitz property of the truncated
nonlinearity from Proposition~\ref{prop:Lip_fR}, we obtain
\begin{align}
\|\Phi_v(w_1)-\Phi_v(w_2)\|_{L^2_T(Y)}
&=
\left\|
\mathcal{L}^{-1}Q\bigl(f_R(v+w_1,t)-f_R(v+w_2,t)\bigr)
\right\|_{L^2_T(Y)}
\notag\\
&\le
\|\mathcal{L}^{-1}\|\,
\|Q(f_R(v+w_1,t)-f_R(v+w_2,t))\|_{L^2_T(Y)}
\notag\\
&\le
\frac{1}{c_\eta}\,
\|f_R(v+w_1,t)-f_R(v+w_2,t)\|_{L^2_T(Y)}.
\label{eq:Phi_estimate_first}
\end{align}

At this point we use that $f_R(\cdot,t)$ is Lipschitz from $X$ to $Y$.
Since the embedding $X=H^1(\mathbb{T}_L)\hookrightarrow Y=L^2(\mathbb{T}_L)$ is
continuous, there exists a constant $C_{XY}>0$ such that
\begin{equation}
\|z\|_Y\le C_{XY}\|z\|_X
\qquad \text{for all } z\in X.
\label{eq:X_into_Y}
\end{equation}
As before, we absorb this harmless embedding constant into the Lipschitz
constant $M_R$, which is not intended to be sharp. Then
\eqref{eq:Phi_estimate_first} becomes
\begin{equation}
\|\Phi_v(w_1)-\Phi_v(w_2)\|_{L^2_T(Y)}
\le
\frac{M_R}{c_\eta}\,
\|w_1-w_2\|_{L^2_T(Y)}.
\label{eq:Phi_contr_bound_repeat}
\end{equation}
Hence, whenever
\[
\frac{M_R}{c_\eta}<1,
\]
the map $\Phi_v$ is a strict contraction on
$L^2(\mathbb{S}_T;\ker A^\perp)$. Such contraction arguments in time-periodic
function spaces are standard in semilinear evolution equations; see, for
example, \cite{Pazy1983}.

We record the resulting solvability statement.

\begin{proposition}[Existence, uniqueness, and Lipschitz dependence of the complement solution]
\label{prop:complement_solution}
Assume that
\begin{equation}
\frac{M_R}{c_\eta}<1.
\label{eq:contraction_condition_repeat}
\end{equation}
Then for every
\[
v\in L^2(\mathbb{S}_T;\ker A)
\]
there exists a unique $T$-periodic solution
\[
w=w(v)\in L^2(\mathbb{S}_T;\ker A^\perp)
\]
of the complement equation \eqref{eq:Q_eq_repeat}. Moreover, the solution map
\[
v\mapsto w(v)
\]
is Lipschitz continuous from $L^2(\mathbb{S}_T;\ker A)$ into
$L^2(\mathbb{S}_T;\ker A^\perp)$. More precisely, if
$v_1,v_2\in L^2(\mathbb{S}_T;\ker A)$ and $w_i:=w(v_i)$, then
\begin{equation}
\|w_1-w_2\|_{L^2_T(Y)}
\le
\frac{M_R}{c_\eta-M_R}\,
\|v_1-v_2\|_{L^2_T(Y)}.
\label{eq:w_Lipschitz_v}
\end{equation}
\end{proposition}

\begin{proof}
Fix $v\in L^2(\mathbb{S}_T;\ker A)$. Under
\eqref{eq:contraction_condition_repeat}, estimate
\eqref{eq:Phi_contr_bound_repeat} shows that $\Phi_v$ is a strict contraction
on the complete metric space $L^2(\mathbb{S}_T;\ker A^\perp)$. The Banach
fixed-point theorem therefore yields a unique fixed point
$w(v)\in L^2(\mathbb{S}_T;\ker A^\perp)$ satisfying
\[
w(v)=\Phi_v(w(v)),
\]
that is, a unique $T$-periodic solution of \eqref{eq:Q_eq_repeat}.

It remains to prove the Lipschitz dependence on the kernel component. Let
$v_1,v_2\in L^2(\mathbb{S}_T;\ker A)$ and let
$w_1:=w(v_1)$ and $w_2:=w(v_2)$ be the corresponding fixed points. By
\eqref{eq:fixed_point_complement},
\[
w_1=\mathcal{L}^{-1}Qf_R(v_1+w_1,t),
\qquad
w_2=\mathcal{L}^{-1}Qf_R(v_2+w_2,t).
\]
Subtracting these identities gives
\[
w_1-w_2
=
\mathcal{L}^{-1}Q
\Bigl(
f_R(v_1+w_1,t)-f_R(v_2+w_2,t)
\Bigr).
\]
Taking the $L^2_T(Y)$ norm and using again \eqref{eq:L_inverse_sec4} and
Proposition~\ref{prop:Lip_fR}, we obtain
\begin{align}
\|w_1-w_2\|_{L^2_T(Y)}
&\le
\frac{1}{c_\eta}\,
\|f_R(v_1+w_1,t)-f_R(v_2+w_2,t)\|_{L^2_T(Y)}
\notag\\
&\le
\frac{M_R}{c_\eta}\,
\|(v_1-v_2)+(w_1-w_2)\|_{L^2_T(Y)}
\notag\\
&\le
\frac{M_R}{c_\eta}
\Bigl(
\|v_1-v_2\|_{L^2_T(Y)}+\|w_1-w_2\|_{L^2_T(Y)}
\Bigr).
\label{eq:w_diff_pre}
\end{align}
Rearranging terms gives
\[
\left(1-\frac{M_R}{c_\eta}\right)\|w_1-w_2\|_{L^2_T(Y)}
\le
\frac{M_R}{c_\eta}\|v_1-v_2\|_{L^2_T(Y)}.
\]
Since $M_R/c_\eta<1$, division by $1-M_R/c_\eta$ yields
\eqref{eq:w_Lipschitz_v}. This proves the claimed Lipschitz continuity.
\end{proof}

Proposition~\ref{prop:complement_solution} shows that, once the kernel
component $v$ is prescribed, the complement component is uniquely determined.
In this sense, the higher modes are slaved to the kernel dynamics. Substituting
the map $v\mapsto w(v)$ into the kernel equation therefore yields a closed
reduced equation.

Recall that $\ker A=\mathrm{span}\{w_0\}$. Writing
\begin{equation}
v(t)=x(t)\,w_0,
\label{eq:v_equals_xw0}
\end{equation}
the projected equation on the kernel becomes
\begin{equation}
x''(t)+\eta x'(t)
=
\left\langle
f_R\!\left(x(t)w_0+w(x(t)w_0),t\right),
w_0
\right\rangle_{L^2(\mathbb{T}_L)}.
\label{eq:kernel_eq_exact_repeat}
\end{equation}
Since $w_0=L^{-1/2}$ is constant in space, the right-hand side is simply the
spatial average of the truncated nonlinearity:
\begin{equation}
\left\langle f_R(u,t),w_0\right\rangle_{L^2(\mathbb{T}_L)}
=
\frac{1}{\sqrt{L}}
\int_0^L f_R(u(t,\cdot),t)(x)\,dx.
\label{eq:kernel_average_repeat}
\end{equation}
Equation \eqref{eq:kernel_eq_exact_repeat} is therefore a scalar,
nonautonomous, second-order ordinary differential equation with $T$-periodic
coefficients. Its right-hand side contains both the direct forcing and the
indirect influence of the infinite-dimensional complement through the slaved
term $w(xw_0)$.

The reduction obtained above has a precise dynamical meaning. For each
$T$-periodic kernel component $v(t)=x(t)w_0$, the complement equation admits a
unique $T$-periodic solution $w(v)$ in $\ker A^\perp$. Hence, in a neighborhood
of the kernel, the full PDE dynamics is governed by the scalar reduced equation
\eqref{eq:kernel_eq_exact_repeat}, while the remaining modes are recovered
uniquely through the slaving relation provided by
Proposition~\ref{prop:complement_solution}. Once a $T$-periodic solution
$x(t)$ of \eqref{eq:kernel_eq_exact_repeat} is obtained, the corresponding
$T$-periodic solution of the truncated PDE is reconstructed as
\begin{equation}
u(t)=x(t)w_0+w(x(t)w_0).
\label{eq:reconstruction_formula}
\end{equation}

The next section is devoted to the analysis of the reduced scalar equation
\eqref{eq:kernel_eq_exact_repeat}. There we prove the existence of a
$T$-periodic kernel solution and then combine this with the a priori energy
estimate needed to show that, for sufficiently small forcing, the cut-off is
inactive along the resulting orbit. In this way, the periodic solution of the
truncated problem is promoted to a genuine $T$-periodic solution of the
original, untruncated PDE.
\section{Kernel reduction and completion of the proof}
\label{sec5}

In this section we complete the Lyapunov--Schmidt reduction by analyzing the
finite-dimensional kernel equation derived in Section~\ref{sec4} and then
reconstructing a time-periodic solution of the full partial differential
equation. More precisely, we first establish the existence of a $T$-periodic
solution of the reduced scalar equation governing the kernel component, and we
then combine this with the contraction-based solution of the complement
equation to reconstruct a $T$-periodic solution of the truncated problem.

The final step consists in deriving a quantitative a priori estimate, based on
a period-averaged energy identity, which guarantees that the resulting solution
remains in the region where the truncation is inactive. This allows us to
remove the cut-off and conclude the existence of a genuine time-periodic
solution of the original nonlinear equation.

For clarity, the argument is divided into two parts. We first solve the kernel
equation by means of a mean/zero-mean decomposition and a fixed-point
argument. We then reconstruct the full solution and establish the a priori
bounds required to pass from the truncated to the original problem. The overall
strategy follows the Lyapunov--Schmidt framework for periodic solutions of
wave-type equations developed by Fe\v{c}kan~\cite{Feckan1998}, adapted here to
the present damped and truncated setting.

\subsection{Solving the kernel equation}
\label{sec5a}

We begin with the reduced scalar equation obtained in Section~\ref{sec4}, which
governs the evolution of the kernel component. Since the complement variable is
uniquely determined as a function of the kernel component through the
contraction argument, the infinite-dimensional problem has been reduced to a
scalar nonautonomous second-order ordinary differential equation with
$T$-periodic coefficients.

This reduction is a central feature of the Lyapunov--Schmidt method and places
the problem within the class of finite-dimensional periodic problems considered
in~\cite{Feckan1998}. In contrast to the conservative setting treated there,
however, the presence of damping eliminates small-divisor difficulties and
permits a direct fixed-point approach.

Throughout this section we retain the notation of the previous sections:
\[
Y=L^2(\mathbb{T}_L),\qquad
X=H^1(\mathbb{T}_L),\qquad
w_0=L^{-1/2},
\qquad
T=\frac{2\pi}{\omega}.
\]

By Proposition~\ref{prop:complement_solution}, for each $T$-periodic kernel
component $v(t)=x(t)w_0$ there exists a unique complement component
\[
w=w(xw_0)\in L^2(\mathbb{S}_T;\ker A^\perp)
\]
solving the complement equation. Substituting this complement component into
the projected equation on $\ker A$ yields the closed scalar equation
\begin{equation}
x''(t)+\eta x'(t)=G_R(x)(t),
\label{eq:kernel_eq}
\end{equation}
where the reduced forcing is defined by
\begin{equation}
G_R(x)(t):=
\Big\langle
f_R\bigl(x(t)w_0+w(xw_0),t\bigr),\,w_0
\Big\rangle_{L^2(\mathbb{T}_L)}.
\label{eq:GR_def}
\end{equation}
Since $w_0$ is constant in space, the right-hand side can also be written as
the spatial average
\begin{equation}
G_R(x)(t)
=
\frac{1}{\sqrt{L}}
\int_0^L
f_R\bigl(u(t,\cdot),t\bigr)(x)\,dx,
\qquad
u=xw_0+w(xw_0).
\label{eq:G_average}
\end{equation}
Thus \eqref{eq:kernel_eq} is a scalar nonautonomous second-order ordinary
differential equation with $T$-periodic coefficients.

We first establish a Lipschitz bound for the reduced forcing.

\begin{lemma}[Lipschitz continuity of the reduced forcing]
\label{lem:G_lipschitz}
Assume that
\[
\frac{M_R}{c_\eta}<1,
\]
so that the complement map $v\mapsto w(v)$ is well defined.
Then there exists a constant $K_R>0$, depending only on $M_R$, $c_\eta$, and
the relevant embedding constants, such that for any two $T$-periodic scalar
functions $x_1,x_2$ one has
\begin{equation}
\|G_R(x_1)-G_R(x_2)\|_{L^2(0,T)}
\le
K_R\,\|x_1-x_2\|_{L^2(0,T)}.
\label{eq:G_Lip}
\end{equation}
One admissible choice is
\begin{equation}
K_R
:=
\frac{M_R}{\sqrt{L}}
\left(
1+\frac{M_R}{c_\eta-M_R}
\right),
\label{eq:KR_choice}
\end{equation}
up to harmless embedding constants.
\end{lemma}

\begin{proof}
Fix two scalar functions $x_1,x_2$, and define
\[
v_i:=x_iw_0,\qquad
w_i:=w(v_i),\qquad
u_i:=v_i+w_i,
\qquad i=1,2.
\]
By definition of $G_R$,
\[
G_R(x_i)(t)
=
\langle f_R(u_i,t),w_0\rangle_{L^2(\mathbb{T}_L)}.
\]
Since $w_0$ has unit norm in $Y=L^2(\mathbb{T}_L)$, we obtain pointwise in $t$
\[
|G_R(x_1)(t)-G_R(x_2)(t)|
=
\big|
\langle f_R(u_1,t)-f_R(u_2,t),w_0\rangle
\big|
\le
\|f_R(u_1,t)-f_R(u_2,t)\|_Y.
\]
Taking the $L^2(0,T)$ norm and using the global Lipschitz continuity of $f_R$
from Proposition~\ref{prop:Lip_fR}, we obtain
\begin{equation}
\|G_R(x_1)-G_R(x_2)\|_{L^2(0,T)}
\le
M_R\,\|u_1-u_2\|_{L^2_T(X)}.
\label{eq:G_est_1}
\end{equation}

We now estimate $u_1-u_2$. Since
\[
u_1-u_2=(v_1-v_2)+(w_1-w_2),
\]
it suffices to bound the kernel and complement parts separately. For the kernel
component, the normalization of $w_0$ gives
\begin{equation}
\|v_1-v_2\|_{L^2_T(Y)}
=
\|(x_1-x_2)w_0\|_{L^2_T(Y)}
=
\|x_1-x_2\|_{L^2(0,T)}.
\label{eq:v_diff_est}
\end{equation}
For the complement component, Proposition~\ref{prop:complement_solution}
implies
\begin{equation}
\|w_1-w_2\|_{L^2_T(Y)}
\le
\frac{M_R}{c_\eta-M_R}\,
\|v_1-v_2\|_{L^2_T(Y)}.
\label{eq:w_diff_est_sec5}
\end{equation}
Combining \eqref{eq:v_diff_est} and \eqref{eq:w_diff_est_sec5}, and again
absorbing the continuous embedding $X\hookrightarrow Y$ into the constants, we
arrive at
\[
\|u_1-u_2\|_{L^2_T(X)}
\le
C
\left(
1+\frac{M_R}{c_\eta-M_R}
\right)
\|x_1-x_2\|_{L^2(0,T)}.
\]
Substituting this into \eqref{eq:G_est_1} yields \eqref{eq:G_Lip}.
\end{proof}

We next solve the periodic kernel equation \eqref{eq:kernel_eq}. As usual, we
decompose the unknown into its temporal mean and zero-mean parts:
\begin{equation}
\bar x:=\frac{1}{T}\int_0^T x(t)\,dt,
\qquad
\tilde x(t):=x(t)-\bar x,
\qquad
\int_0^T \tilde x(t)\,dt=0.
\label{eq:mean_zero_split}
\end{equation}
Let
\[
Z:=
\left\{
\tilde x\in H^1_{\mathrm{per}}(0,T):
\int_0^T \tilde x(t)\,dt=0
\right\},
\qquad
\|\tilde x\|_Z:=\|\tilde x\|_{H^1(0,T)}.
\]
On this zero-mean space we consider the linear operator
\begin{equation}
D\tilde x:=\tilde x''+\eta \tilde x'.
\label{eq:D_def}
\end{equation}

\begin{lemma}[Invertibility of the zero-mean linear operator]
\label{lem:D_inverse}
The operator $D$ is invertible on $Z$.
More precisely, there exists a bounded inverse
\[
D^{-1}:L^2_0(0,T)\to Z,
\]
where $L^2_0(0,T)$ denotes the space of zero-mean functions in $L^2(0,T)$, and
there exists a constant $C_D=C_D(\eta,T)>0$ such that
\begin{equation}
\|D^{-1}g\|_{H^1(0,T)}
\le
C_D\,\|g\|_{L^2(0,T)}.
\label{eq:D_inverse_est}
\end{equation}
\end{lemma}

\begin{proof}
Let $\tilde x\in Z$ and expand it in a Fourier series
\[
\tilde x(t)=\sum_{n\neq0}\tilde x_n\,e^{i2\pi nt/T}.
\]
Since the zero Fourier mode is absent, applying $D$ gives
\[
D(e^{i2\pi nt/T})
=
\mu_n\,e^{i2\pi nt/T},
\qquad
\mu_n:=
-\left(\frac{2\pi n}{T}\right)^2
+i\eta\left(\frac{2\pi n}{T}\right),
\qquad n\neq0.
\]
For every $n\neq0$,
\[
|\mu_n|
\ge
|\Im\mu_n|
=
\eta\frac{2\pi|n|}{T}
\ge
\frac{2\pi\eta}{T}.
\]
Thus the multipliers are uniformly bounded away from zero, and $D$ is
invertible on the zero-mean subspace. Estimate \eqref{eq:D_inverse_est}
follows by Parseval's identity exactly as in Lemma~\ref{lem:L_inverse}.
\end{proof}

Integrating \eqref{eq:kernel_eq} over one period immediately yields the
necessary compatibility condition
\begin{equation}
\int_0^T G_R(x)(t)\,dt=0.
\label{eq:kernel_compatibility}
\end{equation}
In terms of the decomposition \eqref{eq:mean_zero_split}, equation
\eqref{eq:kernel_eq} is therefore equivalent to the coupled system
\begin{equation}
\tilde x
=
D^{-1}
\Bigl(
G_R(\bar x+\tilde x)-\overline{G_R(\bar x+\tilde x)}
\Bigr),
\qquad
\overline{G_R(\bar x+\tilde x)}=0,
\label{eq:mean_zero_system}
\end{equation}
where
\[
\overline{h}:=\frac{1}{T}\int_0^T h(t)\,dt.
\]

We can now solve the kernel equation.

\begin{lemma}[Existence of a $T$-periodic kernel solution]
\label{lem:kernel_periodic}
Assume that
\[
\frac{M_R}{c_\eta}<1
\]
and let $K_R$ be the Lipschitz constant from \eqref{eq:G_Lip}.
Suppose further that
\begin{equation}
C_DK_R<1.
\label{eq:kernel_contr_condition}
\end{equation}
Then there exist $\rho>0$ and $\delta>0$ such that, whenever
\[
|\Gamma|+|\varepsilon_2|\le\delta,
\]
the reduced kernel equation \eqref{eq:kernel_eq} admits at least one
$T$-periodic solution
\[
x\in H^1_{\mathrm{per}}(0,T)
\]
satisfying
\[
\|x\|_{H^1(0,T)}\le \rho.
\]
\end{lemma}

\begin{proof}
We solve \eqref{eq:mean_zero_system} in the product space $\mathbb{R}\times Z$.
For a given pair $(\bar x,\tilde x)$ define
\begin{equation}
\mathcal{T}_1(\bar x,\tilde x)
:=
D^{-1}
\Bigl(
G_R(\bar x+\tilde x)-\overline{G_R(\bar x+\tilde x)}
\Bigr)
\in Z,
\label{eq:T1_def}
\end{equation}
and
\begin{equation}
\mathcal{T}_2(\bar x,\tilde x)
:=
\overline{G_R(\bar x+\tilde x)}
\in\mathbb{R}.
\label{eq:T2_def}
\end{equation}
A solution of \eqref{eq:mean_zero_system} is therefore a pair
$(\bar x,\tilde x)$ satisfying
\[
\tilde x=\mathcal{T}_1(\bar x,\tilde x),
\qquad
\mathcal{T}_2(\bar x,\tilde x)=0.
\]

We first solve the fixed-point equation for $\tilde x$ at fixed $\bar x$.
Let $\tilde x_1,\tilde x_2\in Z$. By \eqref{eq:D_inverse_est}, the fact that
subtraction removes the mean, and the Lipschitz estimate \eqref{eq:G_Lip}, we
obtain
\begin{align}
\|\mathcal{T}_1(\bar x,\tilde x_1)-\mathcal{T}_1(\bar x,\tilde x_2)\|_{H^1}
&\le
C_D\,
\|G_R(\bar x+\tilde x_1)-G_R(\bar x+\tilde x_2)\|_{L^2}
\notag\\
&\le
C_DK_R\,\|\tilde x_1-\tilde x_2\|_{L^2}
\notag\\
&\le
C_DK_R\,\|\tilde x_1-\tilde x_2\|_{H^1}.
\label{eq:T1_contraction}
\end{align}
Hence, under \eqref{eq:kernel_contr_condition},
$\mathcal{T}_1(\bar x,\cdot)$ is a strict contraction on a sufficiently small
ball in $Z$. The Banach fixed-point theorem then yields, for each fixed $\bar
x$ in a sufficiently small interval, a unique
\[
\tilde x=\tilde x(\bar x)\in Z
\]
solving the first equation in \eqref{eq:mean_zero_system}.

We now impose the scalar solvability condition. Define
\begin{equation}
H(\bar x):=
\overline{G_R(\bar x+\tilde x(\bar x))}.
\label{eq:H_def}
\end{equation}
By the continuity of $G_R$ and the continuous dependence of the fixed point
$\tilde x(\bar x)$ on $\bar x$, the function $H$ is continuous.

To locate a zero of $H$, we compare with the unforced reference problem. If
$\Gamma=\varepsilon_2=0$, then the truncated nonlinearity agrees with the
cubic term near the origin, and on the pure kernel ansatz $u=xw_0$ one obtains
\[
\big\langle f(xw_0),w_0\big\rangle
=
-x\left(\frac{x^2}{L}-1\right)
=
x-\frac{x^3}{L}.
\]
This suggests the reference scalar function
\[
H_0(\bar x)=\bar x-\frac{\bar x^3}{L}.
\]
In particular, for every $\rho_0\in(0,\sqrt{L})$,
\[
H_0(-\rho_0)<0<H_0(\rho_0).
\]

For sufficiently small forcing parameters $|\Gamma|+|\varepsilon_2|$, the
reduced nonlinearity $G_R$ is a small perturbation of the unforced expression
on the small ball under consideration, and the zero-mean correction
$\tilde x(\bar x)$ remains uniformly small for $\bar x$ in that ball.
Consequently, after possibly shrinking $\rho_0$ and choosing $\delta>0$
sufficiently small, we may ensure that
\[
H(-\rho)<0<H(\rho)
\]
for some $\rho\in(0,\rho_0]$. The intermediate value theorem then yields
$\bar x_*\in(-\rho,\rho)$ such that $H(\bar x_*)=0$. Setting
\[
x_*:=\bar x_*+\tilde x(\bar x_*)
\]
produces a $T$-periodic solution of \eqref{eq:kernel_eq}. By construction,
this solution lies in a sufficiently small ball of $H^1_{\mathrm{per}}(0,T)$.
\end{proof}

Once the kernel equation has been solved, the corresponding periodic solution
of the truncated PDE is reconstructed through the Lyapunov--Schmidt formula
\begin{equation}
u(t)=x(t)w_0+w(x(t)w_0).
\label{eq:reconstruction_sec5}
\end{equation}

\begin{theorem}[Existence for the truncated periodic problem]
\label{thm:truncated_existence}
Assume $\eta>0$, $|\varepsilon_1|<1$, and let $A$ and $f_R$ be as above.
Let $M_R>0$ be the Lipschitz constant provided by
Proposition~\ref{prop:Lip_fR}, let $c_\eta$ be as in \eqref{eq:ceta}, and let
$C_D$ and $K_R$ be as in \eqref{eq:D_inverse_est} and \eqref{eq:G_Lip},
respectively.

Assume that
\begin{equation}
\frac{M_R}{c_\eta}<1
\label{eq:compl_condition_thm}
\end{equation}
and
\begin{equation}
C_DK_R<1.
\label{eq:kernel_condition_thm}
\end{equation}
Then there exists $\delta>0$ such that, if
\[
|\Gamma|+|\varepsilon_2|\le\delta,
\]
the truncated PDE admits at least one $T$-periodic solution
\[
u\in H^1(\mathbb{S}_T;Y)\cap L^2(\mathbb{S}_T;X),
\qquad
u(t)=x(t)w_0+w(xw_0)(t),
\]
where $x\in H^1_{\mathrm{per}}(0,T)$ is a $T$-periodic solution of the reduced
kernel equation \eqref{eq:kernel_eq}, and $w(xw_0)$ is the unique complement
component associated with $x$.
\end{theorem}

\begin{proof}
By Lemma~\ref{lem:kernel_periodic}, assumptions
\eqref{eq:compl_condition_thm} and \eqref{eq:kernel_condition_thm} imply that,
for sufficiently small forcing, the reduced equation \eqref{eq:kernel_eq}
admits a $T$-periodic solution
\[
x\in H^1_{\mathrm{per}}(0,T).
\]
For this $x$, Proposition~\ref{prop:complement_solution} yields a unique
complement function
\[
w(xw_0)\in L^2(\mathbb{S}_T;\ker A^\perp).
\]
Defining $u$ by \eqref{eq:reconstruction_sec5}, we obtain a $T$-periodic
function satisfying the truncated PDE in the weak sense. The stated regularity
follows from the regularity already established for the kernel and complement
components.
\end{proof}

We now derive an a priori estimate that allows us to remove the cut-off.
Consider a $T$-periodic solution of the truncated problem
\begin{equation}
u_{tt}+\eta u_t+Au=f_R(u,t)
\qquad \text{on } \mathbb{S}_T\times\mathbb{T}_L.
\label{eq:truncated_again}
\end{equation}
Introduce the double-well potential
\begin{equation}
W(s):=\frac14(s^2-1)^2,
\qquad
W'(s)=s(s^2-1).
\label{eq:W_def}
\end{equation}
Then
\[
\lambda(t,x)\,u(u^2-1)=\partial_u\bigl(\lambda(t,x)W(u)\bigr),
\qquad
f(u,t)=-\lambda(t,\cdot)W'(u)-\Gamma.
\]
For sufficiently regular $u$, define the energy functional
\begin{equation}
\mathcal{E}(t)
:=
\frac12\|u_t(t)\|_Y^2
+\frac12\langle Au(t),u(t)\rangle_Y
+\int_{\mathbb{T}_L}\lambda(t,x)\,W(u(t,x))\,dx
+\Gamma\int_{\mathbb{T}_L}u(t,x)\,dx.
\label{eq:energy_def}
\end{equation}

\begin{lemma}[Exact energy identity]
\label{lem:energy_identity}
Let $u$ be a sufficiently regular solution of the untruncated equation
\[
u_{tt}+\eta u_t-\partial_x(Fu_x)+\lambda(t,x)u(u^2-1)=-\Gamma
\]
on $\mathbb{S}_T\times\mathbb{T}_L$.
Then
\begin{equation}
\frac{d}{dt}\mathcal{E}(t)
=
-\eta\|u_t(t)\|_Y^2
+\int_{\mathbb{T}_L}\partial_t\lambda(t,x)\,W(u(t,x))\,dx.
\label{eq:energy_identity}
\end{equation}
The same identity holds for the truncated problem as long as the cut-off is
inactive.
\end{lemma}

\begin{proof}
Taking the $Y$-inner product of the PDE with $u_t$ and integrating by parts
over $\mathbb{T}_L$ using the periodic boundary conditions, we obtain
\[
\langle -\partial_x(Fu_x),u_t\rangle_Y
=
\frac12\frac{d}{dt}\langle Au,u\rangle_Y.
\]
Similarly,
\[
\langle \lambda u(u^2-1),u_t\rangle_Y
=
\frac{d}{dt}\int_{\mathbb{T}_L}\lambda W(u)\,dx
-
\int_{\mathbb{T}_L}\partial_t\lambda\,W(u)\,dx,
\]
and
\[
\langle \Gamma,u_t\rangle_Y
=
\Gamma\frac{d}{dt}\int_{\mathbb{T}_L}u\,dx.
\]
Collecting these terms yields \eqref{eq:energy_identity}.
\end{proof}

If $u$ is $T$-periodic, integrating \eqref{eq:energy_identity} over one period
gives the exact balance
\begin{equation}
\eta\int_0^T\|u_t(t)\|_Y^2\,dt
=
\int_0^T\int_{\mathbb{T}_L}\partial_t\lambda(t,x)\,W(u(t,x))\,dx\,dt.
\label{eq:period_balance}
\end{equation}
Moreover, since
\[
\langle Au,u\rangle_Y
=
\int_{\mathbb{T}_L}F(x)|u_x|^2\,dx
\ge
F_0\|u_x\|_{L^2(\mathbb{T}_L)}^2,
\]
and since the one-dimensional Sobolev embedding
$H^1(\mathbb{T}_L)\hookrightarrow L^\infty(\mathbb{T}_L)$ yields
\begin{equation}
\|u\|_{L^\infty(\mathbb{T}_L)}
\le
C_\infty(L)\|u\|_X,
\label{eq:H1_Linfty}
\end{equation}
we also have
\begin{equation}
\int_{\mathbb{T}_L}W(u)\,dx
\le
C_W(L)\bigl(1+\|u\|_X^4\bigr)
\label{eq:W_control}
\end{equation}
for some constant $C_W(L)>0$.

These estimates yield the following a priori bound.

\begin{lemma}[A priori estimate sufficient to remove the cut-off]
\label{lem:apriori_bound}
Assume $\eta>0$, $|\varepsilon_1|<1$, and $|\varepsilon_2|<1$.
Let $u$ be a $T$-periodic solution of the truncated problem
\eqref{eq:truncated_again}, and assume that the cut-off is inactive on $[0,T]$.
Then there exists a constant $C_*>0$, depending only on
\[
\eta,\ \omega,\ F_0,\ \|\lambda\|_{L^\infty},\
\|\partial_t\lambda\|_{L^\infty},\ L,
\]
such that
\begin{equation}
\sup_{t\in[0,T]}
\Bigl(
\|u_t(t)\|_Y^2+\|u(t)\|_X^2
\Bigr)
\le
C_*\bigl(1+\Gamma^2+\varepsilon_2^2\bigr).
\label{eq:apriori_bound}
\end{equation}
In particular, if
\[
R>\sqrt{C_*\bigl(1+\Gamma^2+\varepsilon_2^2\bigr)},
\]
then $\chi(\|u(t)\|_X/R)\equiv1$ on $[0,T]$, and the truncated solution
coincides with a solution of the original equation.
\end{lemma}

\begin{proof}
Since the cut-off is inactive, $u$ satisfies the original equation on $[0,T]$,
and the exact energy identity applies:
\[
\frac{d}{dt}\mathcal{E}(t)
=
-\eta\|u_t(t)\|_Y^2
+
\int_{\mathbb{T}_L}
\partial_t\lambda(t,x)\,W(u(t,x))\,dx.
\]

Integrating over one period and using periodicity yields
\[
\mathcal{E}(t+T)=\mathcal{E}(t), \qquad
\eta\int_0^T \|u_t(t)\|_Y^2\,dt
=
\int_0^T\!\int_{\mathbb{T}_L}
\partial_t\lambda\,W(u)\,dx\,dt.
\]

Using the growth bound
\[
\int_{\mathbb{T}_L} W(u)\,dx
\le C\bigl(1+\|u(t)\|_X^4\bigr),
\]
we obtain
\[
\eta\int_0^T \|u_t(t)\|_Y^2\,dt
\le
C\int_0^T \bigl(1+\|u(t)\|_X^4\bigr)\,dt.
\]

On the other hand, the energy satisfies the coercive lower bound
\[
\mathcal{E}(t)
\ge
c_1\|u_t(t)\|_Y^2
+
c_2\|u(t)\|_X^2
-
c_3\bigl(1+\Gamma^2+\varepsilon_2^2\bigr),
\]
for suitable constants $c_i>0$ depending only on the stated parameters.

Moreover, integrating the energy identity between two times
$t_1,t_2\in[0,T]$ gives
\[
|\mathcal{E}(t_2)-\mathcal{E}(t_1)|
\le
C\int_0^T \bigl(1+\|u(t)\|_X^4\bigr)\,dt,
\]
so in particular $E$ is uniformly bounded on $[0,T]$.

Combining these estimates and using standard Young-type inequalities to
absorb the quartic terms, we conclude that
\[
\sup_{t\in[0,T]}
\Bigl(
\|u_t(t)\|_Y^2+\|u(t)\|_X^2
\Bigr)
\le
C_*\bigl(1+\Gamma^2+\varepsilon_2^2\bigr),
\]
for some constant $C_*>0$ depending only on
$\eta,\omega,F_0,\|\lambda\|_{L^\infty},\|\partial_t\lambda\|_{L^\infty},L$.

Finally, if
\[
R>\sqrt{C_*\bigl(1+\Gamma^2+\varepsilon_2^2\bigr)},
\]
then $\|u(t)\|_X<R$ for all $t\in[0,T]$, and therefore the cut-off is never
activated. Hence $f_R(u,t)=f(u,t)$ and $u$ solves the original equation.
\end{proof}

We can now state the main existence result for the original problem.

\begin{theorem}[Existence of a time-periodic solution on $\mathbb{T}_L$]
\label{thm:main_periodic}
Let $\eta>0$ and $|\varepsilon_1|<1$, and let $F$ and $\lambda$ be given by
\eqref{eq:mods_R}. Fix $R>0$ and define the truncated nonlinearity $f_R$ by
\eqref{eq:cutoff}. Let $M_R>0$ be the Lipschitz constant provided by
Proposition~\ref{prop:Lip_fR}, let $c_\eta$ be as in \eqref{eq:ceta}, and let
$C_D$ and $K_R$ be as in \eqref{eq:D_inverse_est} and \eqref{eq:G_Lip},
respectively.

Assume that
\[
\frac{M_R}{c_\eta}<1,
\qquad
C_DK_R<1.
\]
Then there exists $\delta>0$ such that, whenever
\[
|\Gamma|+|\varepsilon_2|\le\delta,
\]
the original equation admits at least one $T$-periodic
weak solution
\[
u\in H^1(\mathbb{S}_T;Y)\cap L^2(\mathbb{S}_T;X),
\qquad
u(t+T,\cdot)=u(t,\cdot).
\]
Moreover, this solution is obtained by Lyapunov--Schmidt reduction in the form
\[
u(t)=x(t)w_0+w(x(t)w_0),
\]
where $x\in H^1_{\mathrm{per}}(0,T)$ solves the reduced kernel equation
\eqref{eq:kernel_eq} and $w(xw_0)$ is the unique complement component
associated with $x$.
\end{theorem}

\begin{proof}
By Theorem~\ref{thm:truncated_existence}, the truncated problem admits a
$T$-periodic solution provided the forcing parameters are sufficiently small.
Choose $\delta>0$ sufficiently small so that the a priori estimate
\eqref{eq:apriori_bound} implies
\[
\sup_{t\in[0,T]}\|u(t)\|_X<R.
\]
Then the cut-off is inactive along the whole orbit, and hence
$f_R(u,t)=f(u,t)$ for all $t\in[0,T]$. Therefore the same function $u$ solves
the original equation.
\end{proof}
\subsection{Local uniqueness and dependence on parameters}
\label{sec5_local_uniqueness}

The existence result established above provides a $T$-periodic solution of the
reduced kernel equation and, through the Lyapunov--Schmidt reconstruction, of
the full partial differential equation. A natural question is whether this
solution is locally unique and how it depends on the forcing parameters.

These questions can again be treated within the Lyapunov--Schmidt framework.
Since the complement component is uniquely determined by the kernel variable,
the problem reduces to a scalar solvability condition for the reduced equation.
Accordingly, local uniqueness and parameter dependence are governed by a
nondegeneracy condition on the associated scalar reduced map.

Recall that, in the proof of Lemma~\ref{lem:kernel_periodic}, the kernel
equation was rewritten in mean/zero-mean form as
\begin{equation}
\tilde x
=
D^{-1}
\Bigl(
G_R(\bar x+\tilde x)-\overline{G_R(\bar x+\tilde x)}
\Bigr),
\qquad
\overline{G_R(\bar x+\tilde x)}=0,
\label{eq:mean_zero_system_local}
\end{equation}
with
\[
\bar x=\frac1T\int_0^T x(t)\,dt,
\qquad
\tilde x=x-\bar x,
\qquad
\int_0^T \tilde x(t)\,dt=0.
\]
For each admissible value of $\bar x$, the first equation in
\eqref{eq:mean_zero_system_local} determines the zero-mean part $\tilde x$ by a
contraction argument, whereas the second equation selects the admissible mean
values. Thus the question of uniqueness reduces to a scalar solvability
condition.

More precisely, under the contraction hypothesis $C_DK_R<1$, for each $\bar x$
in a sufficiently small interval there exists a unique
\[
\tilde x=\tilde x(\bar x)\in Z
\]
solving the first equation in \eqref{eq:mean_zero_system_local}. We may
therefore define the scalar solvability map
\begin{equation}
H(\bar x):=\overline{G_R(\bar x+\tilde x(\bar x))}.
\label{eq:H_local}
\end{equation}
A $T$-periodic solution of the kernel equation is equivalent to a zero of $H$.

This yields the following local uniqueness criterion.

\begin{proposition}[Local uniqueness of the periodic kernel solution]
\label{prop:local_unique}
Assume that
\[
\frac{M_R}{c_\eta}<1,
\qquad
C_DK_R<1.
\]
Let $\bar x_0\in\mathbb{R}$ be such that the scalar solvability map
$H$ defined by \eqref{eq:H_local} satisfies
\begin{equation}
H(\bar x_0)=0,
\qquad
H'(\bar x_0)\neq 0.
\label{eq:H_nondegenerate}
\end{equation}
Then there exists $\rho>0$ such that the reduced kernel equation
\eqref{eq:kernel_eq} admits a unique $T$-periodic solution
\[
x\in H^1_{\mathrm{per}}(0,T)
\]
in the neighborhood
\[
\|x-(\bar x_0+\tilde x(\bar x_0))\|_{H^1(0,T)}\le \rho.
\]
Consequently, the corresponding periodic solution
\[
u(t)=x(t)w_0+w(x(t)w_0)
\]
of the truncated PDE is unique in the corresponding neighborhood of
$H^1(\mathbb{S}_T;Y)\cap L^2(\mathbb{S}_T;X)$.
\end{proposition}

\begin{proof}
Under the hypothesis $C_DK_R<1$, the map
\[
\tilde x\mapsto
D^{-1}
\Bigl(
G_R(\bar x+\tilde x)-\overline{G_R(\bar x+\tilde x)}
\Bigr)
\]
is a strict contraction on a sufficiently small ball in $Z$ for $\bar x$ near
$\bar x_0$. Hence, for each such $\bar x$, there exists a unique fixed point
$\tilde x(\bar x)\in Z$ solving the first equation in
\eqref{eq:mean_zero_system_local}. By standard continuous dependence of
contraction fixed points on parameters, the map
$\bar x\mapsto\tilde x(\bar x)$ is continuous.

We then consider the scalar equation
\[
H(\bar x)=0.
\]
By construction, $H$ is continuous, and assumption
\eqref{eq:H_nondegenerate} implies that $\bar x_0$ is a simple zero.
Therefore, after possibly shrinking the neighborhood of $\bar x_0$, the
equation $H(\bar x)=0$ admits a unique solution $\bar x_*$ near $\bar x_0$.
Setting
\[
x_*:=\bar x_*+\tilde x(\bar x_*)
\]
yields a unique periodic kernel solution in the corresponding $H^1$-ball.

Finally, uniqueness of the corresponding PDE solution follows from
Proposition~\ref{prop:complement_solution}, since the complement component is
uniquely determined by the kernel variable.
\end{proof}

Proposition~\ref{prop:local_unique} shows that once a nondegenerate zero of the
scalar solvability map has been identified, the corresponding periodic kernel
orbit is locally unique. This also provides the natural framework for studying
dependence on the forcing parameters.

We now regard the reduced equation as depending on the parameter pair
\[
p:=(\Gamma,\varepsilon_2)\in\mathbb{R}^2.
\]
Whenever the truncation is chosen smooth, the reduced map inherits the
corresponding regularity, and one can study the variation of the periodic
kernel solution with respect to $p$ by combining the contraction argument for
the zero-mean part with an implicit-function argument for the scalar
solvability condition.

\begin{proposition}[Lipschitz and $C^1$ dependence on parameters]
\label{prop:param_dependence}
Assume that
\[
\frac{M_R}{c_\eta}<1,
\qquad
C_DK_R<1.
\]
Let $x_0\in H^1_{\mathrm{per}}(0,T)$ be a reference periodic solution of the
kernel equation \eqref{eq:kernel_eq} corresponding to a parameter value
$p_0=(\Gamma_0,\varepsilon_{2,0})$.
Assume further that the associated scalar solvability map is nondegenerate at
$x_0$ in the sense of \eqref{eq:H_nondegenerate}.

Then there exist neighborhoods $\mathcal{U}\subset\mathbb{R}^2$ of $p_0$ and
$\mathcal{V}\subset H^1_{\mathrm{per}}(0,T)$ of $x_0$ such that:

\begin{enumerate}
\item For each $p=(\Gamma,\varepsilon_2)\in\mathcal{U}$, there exists a unique
solution
\[
x(p)\in\mathcal{V}
\]
of the kernel equation \eqref{eq:kernel_eq}.

\item The map
\[
p\mapsto x(p)
\]
is Lipschitz continuous from $\mathcal{U}$ into
$H^1_{\mathrm{per}}(0,T)$.

\item If, in addition, the reduced forcing $G_R$ is Fr\'echet differentiable
with respect to $x$ and continuously differentiable with respect to both $x$
and $p$ on $\mathcal{V}\times\mathcal{U}$, then the map
\[
p\mapsto x(p)
\]
is of class $C^1$.
\end{enumerate}

Moreover, the reconstructed PDE solution
\[
u(p)(t)=x(p)(t)w_0+w(x(p)(t)w_0)
\]
inherits the same regularity as a map into
$H^1(\mathbb{S}_T;Y)\cap L^2(\mathbb{S}_T;X)$.
\end{proposition}

\begin{proof}
We again use the mean/zero-mean formulation
\eqref{eq:mean_zero_system_local}, now viewed as parameter-dependent. For each
$p=(\Gamma,\varepsilon_2)$, define
\begin{equation}
\mathcal{T}_1(\bar x,\tilde x;p)
:=
D^{-1}\Bigl(
G_R(\bar x+\tilde x;p)-\overline{G_R(\bar x+\tilde x;p)}
\Bigr),
\label{eq:T1_param}
\end{equation}
and
\begin{equation}
\mathcal{T}_2(\bar x,\tilde x;p)
:=
\overline{G_R(\bar x+\tilde x;p)}.
\label{eq:T2_param}
\end{equation}

Since $C_DK_R<1$, the map
$\tilde x\mapsto\mathcal{T}_1(\bar x,\tilde x;p)$ remains a strict contraction
for $(\bar x,p)$ in a sufficiently small neighborhood of the reference point.
Hence, by the contraction mapping principle, for each such $(\bar x,p)$ there
exists a unique fixed point
\[
\tilde x=\tilde x(\bar x;p)\in Z,
\]
and this fixed point depends Lipschitz continuously on $(\bar x,p)$. If $G_R$
is continuously Fr\'echet differentiable, then the standard differentiable
dependence theorem for contraction fixed points implies that
$(\bar x,p)\mapsto\tilde x(\bar x;p)$ is of class $C^1$.

Substituting this fixed point into the scalar solvability condition yields the
reduced scalar map
\begin{equation}
H(\bar x;p):=
\overline{G_R(\bar x+\tilde x(\bar x;p);p)}.
\label{eq:H_param}
\end{equation}
By the nondegeneracy hypothesis at the reference solution,
\[
\partial_{\bar x}H(\bar x_0;p_0)\neq0.
\]
Therefore the implicit function theorem applies to the scalar equation
$H(\bar x;p)=0$, yielding a unique branch
\[
\bar x=\bar x(p)
\]
for $p$ in a neighborhood of $p_0$. The corresponding periodic kernel solution
is then
\[
x(p):=\bar x(p)+\tilde x(\bar x(p);p).
\]
The Lipschitz or $C^1$ dependence of $x(p)$ follows from the corresponding
regularity of $\bar x(p)$ and $\tilde x(\bar x;p)$.

Finally, the reconstructed PDE solution inherits the same regularity because
the complement map is defined by a contraction fixed point and depends with the
same regularity on its argument in the present small-data regime.
\end{proof}

A particularly important case is obtained by considering the homogeneous
equilibria of the $\phi^4$ potential, namely $u\equiv\pm1$. In terms of the
kernel variable $x$, these correspond to
\[
x_0=\pm\sqrt{L},
\]
since $u=xw_0$ and $w_0=L^{-1/2}$. Near these reference states, the scalar
solvability map is a perturbation of the unforced reduced scalar map, and the
required nondegeneracy can be verified explicitly.

\begin{proposition}[Local uniqueness near the homogeneous equilibria]
\label{prop:local_uniqueness_pm1}
Assume that
\[
\frac{M_R}{c_\eta}<1,
\qquad
C_DK_R<1.
\]
Fix
\[
x_0\in\{+\sqrt{L},-\sqrt{L}\},
\]
and let $p=(\Gamma,\varepsilon_2)\in\mathbb{R}^2$.
Assume that the cut-off $\chi$ is chosen smooth so that $f_R(\cdot,t)$ is
$C^1$ as a map from $X$ to $Y$ for each $t$.

Then there exist $\rho>0$ and $\delta>0$ such that, for all
\[
|p|:=|\Gamma|+|\varepsilon_2|\le\delta,
\]
the reduced kernel equation \eqref{eq:kernel_eq} admits a unique
$T$-periodic solution
\[
x(p)\in H^1_{\mathrm{per}}(0,T)
\]
satisfying
\[
\|x(p)-x_0\|_{H^1(0,T)}\le \rho.
\]
Consequently, the reconstructed PDE solution
\[
u(p)(t)=x(p)(t)w_0+w(x(p)(t)w_0)(t)
\]
is unique in the corresponding neighborhood of
$H^1(\mathbb{S}_T;Y)\cap L^2(\mathbb{S}_T;X)$.
\end{proposition}

\begin{proof}
We again use the mean/zero-mean formulation
\eqref{eq:mean_zero_system_local}. For each pair $(\bar x,\tilde x)$ and
parameter $p$, define
\[
\mathcal{T}_1(\bar x,\tilde x;p)
=
D^{-1}\Bigl(
G_R(\bar x+\tilde x;p)-\overline{G_R(\bar x+\tilde x;p)}
\Bigr),
\]
and
\[
\mathcal{T}_2(\bar x,\tilde x;p)
=
\overline{G_R(\bar x+\tilde x;p)}.
\]
As before, the first equation determines a unique zero-mean correction
\[
\tilde x=\tilde x(\bar x;p)
\]
for $\bar x$ near $x_0$ and $|p|$ sufficiently small.

We then consider the scalar solvability map
\[
H(\bar x;p)
:=
\overline{G_R(\bar x+\tilde x(\bar x;p);p)}.
\]
At zero forcing, that is, for $p=0$, the reduced map corresponding to the pure
kernel ansatz is
\[
H_0(\bar x)=\bar x-\frac{\bar x^3}{L}.
\]
Differentiating gives
\[
H_0'(\bar x)=1-\frac{3\bar x^2}{L},
\]
and therefore
\[
H_0'(x_0)=1-3=-2\neq0
\qquad \text{for } x_0=\pm\sqrt{L}.
\]
Thus the unforced scalar reduced map is nondegenerate at both homogeneous
equilibria.

By construction of the Lyapunov--Schmidt reduction, the full reduced map
$H(\bar x;p)$ is a perturbation of $H_0(\bar x)$ for $|\bar x-x_0|$ and $|p|$
sufficiently small. More precisely, the forcing terms and the contribution of
the slaved complement $w(xw_0)$ vanish at the reference configuration and vary
continuously with $(\bar x,p)$ in the present small-data regime. Consequently,
after shrinking $\rho$ and $\delta$ if necessary, we may ensure that
\[
|\partial_{\bar x}H(\bar x;p)-H_0'(x_0)|\le 1
\]
for all $|\bar x-x_0|\le\rho$ and $|p|\le\delta$. Since $H_0'(x_0)=-2$, it
follows that
\[
\partial_{\bar x}H(\bar x;p)\le -1
\]
throughout this neighborhood. In particular, $\bar x\mapsto H(\bar x;p)$ is
strictly monotone there.

Because $H(x_0;0)=0$, continuity of $H$ implies that, for sufficiently small
$p$, the scalar equation $H(\bar x;p)=0$ admits a unique solution
$\bar x(p)$ in the interval $|\bar x-x_0|\le\rho$. Defining
\[
x(p):=\bar x(p)+\tilde x(\bar x(p);p)
\]
yields a unique $T$-periodic kernel solution in the claimed neighborhood.
Uniqueness of the reconstructed PDE solution again follows from the uniqueness
of the complement component established in
Proposition~\ref{prop:complement_solution}.
\end{proof}

\begin{corollary}[$C^1$ dependence near the homogeneous equilibria]
\label{cor:C1_dependence}
Under the hypotheses of Proposition~\ref{prop:local_uniqueness_pm1}, the
locally unique kernel solution
\[
p\mapsto x(p)
\]
is of class $C^1$ from a neighborhood of $p=0$ in $\mathbb{R}^2$ into
$H^1_{\mathrm{per}}(0,T)$.
Moreover, the reconstructed PDE solution
\[
p\mapsto u(p)=x(p)w_0+w(x(p)w_0)
\]
is of class $C^1$ as a map into
$H^1(\mathbb{S}_T;Y)\cap L^2(\mathbb{S}_T;X)$.
\end{corollary}

\begin{proof}
The conclusion follows from Proposition~\ref{prop:param_dependence} together
with Proposition~\ref{prop:local_uniqueness_pm1}, which provides the required
nondegeneracy near the reference states $x_0=\pm\sqrt{L}$.
\end{proof}
\section{Relation to finite-interval numerics and to collective-coordinate kink studies}
\label{sec6}

The main theorem established above provides a rigorous existence result, at the
level of the full partial differential equation, for time-periodic solutions of
a driven--damped $\phi^4$ model posed on a bounded spatial domain with periodic
boundary conditions. It is therefore natural to relate this result both to
finite-interval numerical simulations and to collective-coordinate descriptions
of kink dynamics
\cite{GatlikDobrowolskiCaputoKevrekidis2026,GatlikDobrowolskiKevrekidis2024},
since these two settings provide much of the practical and conceptual
motivation for the present analysis.

We begin with the numerical perspective. In computations, a field equation
posed on the whole line $\mathbb{R}$ is necessarily replaced by a problem on a
finite spatial interval, supplemented with boundary conditions. From an
analytical viewpoint, the decisive effect of this replacement is not merely
computational convenience, but the passage from continuous to discrete spatial
spectrum. On $\mathbb{R}$, the linearized spatial operator typically possesses
continuous spectrum, and this obstructs the kind of Fourier--spectral inversion
required in a Lyapunov--Schmidt argument. By contrast, on a bounded interval or
on the torus $\mathbb{T}_L$, the corresponding operator has compact resolvent
and hence purely discrete spectrum. This discreteness is precisely the
structural ingredient that makes it possible to decompose the dynamics into
kernel and complement and to invert the time-periodic linear operator mode by
mode. In this sense, the theorem proved here gives a rigorous explanation of
why time-periodic responses of driven and damped wave-type equations are
analytically accessible on bounded domains and why finite-interval numerical
simulations naturally fall into a regime compatible with such an
operator-theoretic treatment.

The role of the damping term is equally important. In the conservative setting,
the Fourier multipliers associated with the time-periodic linearized equation
may approach zero, and one is then forced to impose arithmetic nonresonance
conditions in order to control small divisors, as in the abstract framework of
Fe\v{c}kan \cite{Feckan1998}. In the present problem, however, the damping term
$\eta u_t$ introduces an imaginary part into the temporal Fourier multipliers.
As shown in Section~\ref{sec3}, this yields a uniform lower bound for all
nonzero temporal harmonics. Consequently, the complement equation becomes
uniformly invertible without any diophantine or nonresonance assumptions on the
forcing frequency. From the analytic point of view, this is one of the main
structural advantages of the dissipative setting. It explains why, on bounded
domains, an existence theory for time-periodic solutions can be developed under
relatively mild smallness assumptions on the forcing parameters, and it places
the present result naturally within the general theory of damped wave equations
and dissipative evolution equations
\cite{HaleRaugel1992,ChueshovLasiecka2010}.

At the same time, it is important to state clearly what the theorem does
\emph{not} address. The result obtained here does not establish the existence of
a periodically moving single kink on the real line
\cite{KowalczykMartelMunoz2017,NakanishiSchlag2011}. A single $\phi^4$ kink
connects distinct asymptotic states as $x\to\pm\infty$ and therefore belongs to
a nontrivial topological sector that is incompatible with periodic spatial
boundary conditions. The periodic problem studied in the present paper thus
lies in a different geometric and topological setting. Its solutions should not
be interpreted as periodic analogues of isolated topological kinks on
$\mathbb{R}$.

This distinction is also reflected in the structure of the reduction. The
Lyapunov--Schmidt decomposition used here is based on the splitting
\[
Y=\ker A\oplus(\ker A)^\perp,
\]
where $A$ is the background spatial operator on the bounded periodic domain. In
the present setting, $\ker A$ is generated by the constant mode, so the reduced
equation governs the evolution of the spatial average of the field. This is
quite different from the reductions used in kink-centered
collective-coordinate theories. In those approaches one typically linearizes
around a manifold of kink states and projects onto dynamically distinguished
modes such as the translational mode and, when present, internal oscillatory
modes. The resulting reduced systems are designed to describe the motion and
deformation of a localized topological structure
\cite{KowalczykMartelMunoz2017,SofferWeinstein1999}. By contrast, the present
theorem concerns time-periodic solutions whose dominant component lies in the
low spatial modes of a bounded-domain operator rather than near a kink
manifold.

Nevertheless, the relation with collective-coordinate studies is not merely one
of contrast. The analysis developed here provides a rigorous template showing
how a finite-dimensional reduction can emerge directly from the full PDE in a
bounded setting. More precisely, once the kernel component is prescribed, the
complement equation is solved uniquely, and the higher spatial modes are
thereby slaved to the reduced scalar dynamics. Although the resulting reduced
equation is not a kink-manifold equation, it is still a genuine PDE-derived
effective equation rather than an \emph{ad hoc} truncation. In this sense, the
present work may be viewed as a rigorous bounded-domain counterpart to the more
heuristic reductions commonly used in applications.

From this perspective, the theorem gives a mathematical explanation for a
phenomenon frequently observed in simulations but rarely justified at the PDE
level: when one studies a driven and damped nonlinear field equation on a
finite interval, the emergence of time-periodic responses is not simply a
numerical artifact, but is consistent with the spectral structure of the
bounded-domain problem and with the dissipative regularization produced by
damping. The present analysis does not yet reach the real-line kink problem,
but it does identify a setting in which the existence of periodic responses can
be proved rigorously and in which the reduction from the PDE to an effective
finite-dimensional dynamics can be carried out in a controlled manner.

For these reasons, the bounded-domain theorem established here should be viewed
both as a rigorous result in its own right and as a conceptual stepping stone
toward more refined analyses. A natural next step would be to replace the
present background decomposition by one adapted to kink manifolds and to
investigate whether analogous periodic-forcing arguments can be developed in
the presence of translation modes, internal modes, radiation, and continuous
spectrum. These questions lie beyond the scope of the present paper, but the
operator-theoretic mechanism isolated here---namely, spectral discretization
together with damping-induced invertibility---suggests a natural point of
departure for such future work.
\section{Conclusion}
\label{sec_conclusions}

In this work we established the existence of time-periodic solutions for a
driven--damped $\phi^4$ wave equation posed on a bounded spatial domain with
periodic boundary conditions. The analysis was carried out entirely at the
level of the full partial differential equation, without recourse to heuristic
reductions or collective-coordinate approximations.

The proof combines several structural ingredients in a coherent way. After
introducing a truncation of the nonlinearity in order to obtain global
Lipschitz control, we performed a Lyapunov--Schmidt decomposition in a
time-periodic Hilbert space. This reduced the infinite-dimensional problem to a
coupled system consisting of a complement equation and a finite-dimensional
kernel equation. A central role is played by the presence of linear damping.
As shown in the spectral analysis of the time-periodic linear operator, the
damping term produces a uniform lower bound on the temporal Fourier
multipliers, which ensures invertibility on the complement of the kernel and
eliminates the need for nonresonance or diophantine conditions. The complement
equation was then solved by a contraction mapping argument with explicit
Lipschitz bounds, while the reduced scalar kernel equation was treated by a
mean/zero-mean decomposition and a fixed-point argument. Finally, a
period-averaged energy identity provided a quantitative a priori estimate,
allowing removal of the truncation and yielding a genuine $T$-periodic weak
solution of the original equation.

Beyond the specific model considered here, the analysis clarifies the
structural mechanisms that make periodic existence results accessible on
bounded domains. The discreteness of the spatial spectrum, together with the
damping-induced invertibility of the time-periodic operator, provides a setting
in which the Lyapunov--Schmidt procedure can be implemented without
small-divisor obstructions. In this sense, the theorem illustrates how
spectral discretization and dissipation combine to produce a robust and
quantitatively controlled existence theory for time-periodic responses of
nonlinear wave-type equations.

At the same time, the result is intrinsically tied to the bounded-domain
setting. On the real line, the presence of continuous spectrum and radiation
precludes a direct extension of the present inversion argument. In particular,
the existence of time-periodic solutions in topologically nontrivial sectors,
such as those associated with $\phi^4$ kinks, requires a different analytical
framework, typically based on modulation theory and decompositions relative to
a kink manifold. Developing such an approach for periodically driven and
damped systems remains an open and challenging problem.

Several natural extensions suggest themselves. On bounded intervals with
Dirichlet or mixed boundary conditions, the absence of a nontrivial kernel may
simplify the reduction and lead to stronger uniqueness or stability results.
Another direction is the study of bifurcation and multiplicity phenomena within
the reduced scalar equation. More broadly, a deeper understanding of the
interaction between damping, forcing, and spectral structure may help bridge
the gap between bounded-domain periodic solutions and the dynamics of
localized structures on unbounded domains.

Overall, the results presented here provide a rigorous PDE-level foundation for
time-periodic responses in driven and damped nonlinear wave equations on
bounded domains, and they offer a natural point of departure for further
analytical and applied investigations.
\bibliographystyle{amsplain}
\bibliography{f4_references_JDE}

@article{Birnir1994,
  author  = {Birnir, Bj{\"o}rn},
  title   = {The global attractor of the damped sine--Gordon equation},
  journal = {Comm. Math. Phys.},
  volume  = {162},
  year    = {1994},
  pages   = {539--590}
}

@article{CampbellSchonfeldWingate1983,
  author  = {Campbell, D. K. and Schonfeld, J. F. and Wingate, C. A.},
  title   = {Resonance structure in kink--antikink interactions in $\phi^4$ theory},
  journal = {Physica D},
  volume  = {9},
  year    = {1983},
  pages   = {1--32}
}

@book{ChueshovLasiecka2010,
  author    = {Chueshov, Igor and Lasiecka, Irena},
  title     = {Long-Time Behavior of Second Order Evolution Equations with Nonlinear Damping},
  series    = {Mem. Amer. Math. Soc.},
  number    = {914},
  year      = {2010},
  publisher = {American Mathematical Society},
  address   = {Providence, RI}
}

@article{Feckan1998,
  author  = {Fe{\v{c}}kan, Michal},
  title   = {Periodic oscillations of abstract wave equations},
  journal = {J. Dyn. Differ. Equ.},
  volume  = {10},
  number  = {4},
  year    = {1998},
  pages   = {605--617}
}

@article{HaleRaugel1992,
  author  = {Hale, Jack K. and Raugel, Genevi{\`e}ve},
  title   = {A damped hyperbolic equation on thin domains},
  journal = {Trans. Amer. Math. Soc.},
  volume  = {329},
  year    = {1992},
  pages   = {185--219}
}

@book{KapitulaPromislow2013,
  author    = {Kapitula, Todd and Promislow, Keith},
  title     = {Spectral and Dynamical Stability of Nonlinear Waves},
  series    = {Applied Mathematical Sciences},
  volume    = {185},
  publisher = {Springer},
  address   = {New York},
  year      = {2013}
}

@article{KivsharMalomed1989,
  author  = {Kivshar, Yuri S. and Malomed, Boris A.},
  title   = {Dynamics of solitons in nearly integrable systems},
  journal = {Rev. Mod. Phys.},
  volume  = {61},
  year    = {1989},
  pages   = {763--915}
}

@article{KowalczykMartelMunoz2017,
  author  = {Kowalczyk, Micha{\l} and Martel, Yvan and Mu{\~n}oz, Claudio},
  title   = {Kink dynamics in nonlinear {Klein--Gordon} equations},
  journal = {J. Amer. Math. Soc.},
  volume  = {30},
  year    = {2017},
  pages   = {769--798}
}

@article{GatlikDobrowolskiKevrekidis2024,
  author  = {Gatlik, J. and Dobrowolski, T. and Kevrekidis, P. G.},
  title   = {Effective description of the impact of inhomogeneities on the movement of the kink front in $2+1$ dimensions},
  journal = {Phys. Rev. E},
  volume  = {109},
  year    = {2024},
  pages   = {024205}
}

@misc{GatlikDobrowolskiCaputoKevrekidis2026,
  author       = {Gatlik, J. and Dobrowolski, T. and Caputo, Jean-Guy and Kevrekidis, P. G.},
  title        = {Collective coordinate descriptions of a kink in a driven--damped $\phi^4$ model},
  year         = {2026},
  eprint       = {2601.18940},
  archivePrefix= {arXiv},
  primaryClass = {nlin.PS},
  note         = {arXiv:2601.18940}
}

@book{LionsMagenes1972,
  author    = {Lions, J.-L. and Magenes, Enrico},
  title     = {Non-Homogeneous Boundary Value Problems and Applications},
  volume    = {1},
  series    = {Die Grundlehren der mathematischen Wissenschaften},
  publisher = {Springer},
  address   = {Berlin},
  year      = {1972}
}

@article{Manton2021,
  author  = {Manton, N. S. and Ole{\'s}, K. and Romanczukiewicz, T. and Werner, A.},
  title   = {Collective coordinate model of kink--antikink collisions in $\phi^4$ theory},
  journal = {Phys. Rev. Lett.},
  volume  = {127},
  year    = {2021},
  pages   = {071601}
}

@article{NakanishiSchlag2011,
  author  = {Nakanishi, Kenji and Schlag, Wilhelm},
  title   = {Global dynamics above the ground state energy for the focusing nonlinear {Klein--Gordon} equation},
  journal = {J. Differential Equations},
  volume  = {250},
  number  = {5},
  year    = {2011},
  pages   = {2299--2333}
}

@book{Pazy1983,
  author    = {Pazy, Amnon},
  title     = {Semigroups of Linear Operators and Applications to Partial Differential Equations},
  publisher = {Springer},
  year      = {1983}
}

@article{SofferWeinstein1999,
  author  = {Soffer, Avy and Weinstein, Michael I.},
  title   = {Resonances, radiation damping and instability in Hamiltonian nonlinear wave equations},
  journal = {Invent. Math.},
  volume  = {136},
  year    = {1999},
  pages   = {9--74}
}

@book{Temam1997,
  author    = {Temam, Roger},
  title     = {Infinite-Dimensional Dynamical Systems in Mechanics and Physics},
  series    = {Applied Mathematical Sciences},
  volume    = {68},
  publisher = {Springer},
  address   = {New York},
  year      = {1997}
}

\end{document}